\documentclass[11pt, oneside]{article}

\usepackage{amsmath, amssymb, amsthm}
\usepackage{geometry}
\usepackage{epsfig}
\usepackage{xcolor}
\usepackage[font=small,labelfont=bf, width=.7\textwidth]{caption}

\newtheorem{theorem}{Theorem}
\newtheorem{question}[theorem]{Question}

\newtheorem{lemma}[theorem]{Lemma}

\newtheorem{claim}{Claim}

\newcommand{\lfrac}[2]{#1/#2}
\newcommand{\sizeof}[1]{\left\lvert{#1}\right\rvert}
\newenvironment{proofc}{\begin{proof}[Proof of Claim]}{\end{proof}}

\begin{document}

\title{List-edge-colouring planar graphs with precoloured edges}
\author{
Joshua Harrelson\footnote{jth0048@auburn.edu}
\qquad
Jessica McDonald\footnote{mcdonald@auburn.edu; Supported in part by NSF grant DMS-1600551.}
\qquad
Gregory J. Puleo\footnote{gjp0007@auburn.edu}\\
\medskip\\
Department of Mathematics and Statistics\\
Auburn University\\
Auburn, AL 36849
}

\date{}

\maketitle

\begin{abstract} Let $G$ be a simple planar graph of maximum degree
  $\Delta$, let $t$ be a positive integer, and let $L$ be an edge list assignment on $G$
  with $\sizeof{L(e)} \geq \Delta+t$ for all $e \in E(G)$. We prove that if  $H$ is a subgraph of $G$ that
  has been $L$-edge-coloured, then the edge-precolouring can be
  extended to an $L$-edge-colouring of $G$, provided that $H$ has maximum degree $d\leq t$ and either $d \leq t-4$ or $\Delta$ is large enough
  ($\Delta \geq 16+d$ suffices). If $d>t$, there
  are examples for any choice of $\Delta$ where the extension is
  impossible.
\end{abstract}

\section{Introduction}

In this paper all graphs are simple.

An \emph{edge-colouring} of $G$ is an assignment of colours to the edges of $G$ so that adjacent edges receive different colours; if at most $k$ colours are used we say it is a \emph{$k$-edge-colouring}. The \emph{chromatic index} of $G$, denoted $\chi'(G)$, is the minimum $k$ such that $G$ is $k$-edge-colourable. It is obvious that $\chi'(G)\geq \Delta$, where $\Delta:=\Delta(G)$ is the maximum degree of $G$, and Vizing's Theorem \cite{V1} says that $\chi'(G)\leq \Delta+1$.

In this paper we are looking to edge-colour a graph $G$, but with the constraint that some edges have already been coloured and cannot be changed.  In this scenario we have no control over the edge-precolouring -- if the edge-precoloured subgraph is $H$, then it will certainly have at least $\chi'(H)$ colours, but it could have many more, perhaps even more than $\chi'(G)$ colours. If we are looking to extend the edge-precolouring to a $k$-edge-colouring of $G$, then we will certainly need that $k$ is at least the maximum degree of $G$, and that the edge-colouring of $H$ uses at most $k$ colours (i.e. is a $k$-edge-colouring).  In general we consider the following question, first posed by Marcotte and Seymour \cite{MS}:

\begin{question}\label{qu} Given a graph $G$ with maximum degree $\Delta$ and a subgraph $H$ of $G$ that has been $(\Delta+t)$-edge-coloured, can the edge-precolouring of $H$ be extended to a $(\Delta+t)$-edge-colouring of $G$?\end{question}

Marcotte and Seymour's main result in \cite{MS} is a necessary
condition for the answer to Question \ref{qu} to be ``yes''; they
prove that this condition is also sufficient when $G$ is a multiforest
(the condition is rather technical, so we do not state it here).
Question \ref{qu} was shown to be NP-complete by Colbourn \cite{Col},
and Marx \cite{Marx} showed that this is true even when $G$ is a
planar 3-regular bipartite graph. Since, as Holyer~\cite{Ho} showed,
it is NP-complete to decide whether $\chi'(G) = \Delta(G)$ or not, the
special case $t=0$ of Question~\ref{qu} is also NP-complete for
general graphs. In this paper we focus on Question \ref{qu} for planar
graphs. Before saying more about planar graphs in particular however,
let us make several quick observations about Question \ref{qu} in
general.

Firstly, if $t$ is huge -- say at least $\Delta-1$ -- then the answer
is \emph{yes}, and moreover, the extension can be done greedily. This
is because an edge in $G$ sees at most $2(\Delta-1)$ other edges, and
when $t\geq \Delta-1$, this value is at most $\Delta+t-1$. If the
maximum degree of $H$ is $\Delta$ then this threshold for $t$ is
actually sharp. To see this, consider the graph $G$ shown in Figure
\ref{tlessthankminus1}, formed by taking a copy of $K_{1, \Delta}$
with one edge coloured $\Delta$ and the rest uncoloured, and joining
each leaf to $\Delta-1$ distinct new vertices via edges coloured
$1, 2, \ldots, \Delta-1$. Then $G$ has maximum degree $\Delta$, as
does its edge-precoloured subgraph. However, in order to extend the
edge-precolouring to a $(\Delta+t)$-edge-colouring of $G$, we need
$\Delta-1$ new colours, which forces $t \geq \Delta-1$.

\begin{figure}
\centering
\includegraphics[height=4cm]{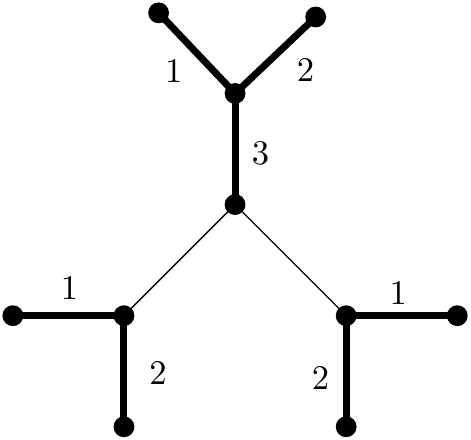}
\caption{A graph $G$ with maximum degree $\Delta=3$ with a precoloured subgraph of maximum degree $\Delta$. In order to extend the edge-precolouring to a $(\Delta+t)$-edge-colouring of $G$ we need $t\geq \Delta-1$.}
\label{tlessthankminus1}
\end{figure}

Given the above paragraph, Question \ref{qu} is only interesting when
the maximum degree of $H$, say $d$, is strictly less than
$\Delta$. Here, we get a natural barrier to extension when $d>t$, via
nearly the same example as above. Let $G$ be the graph shown in Figure
\ref{tlessthank}, formed by taking an (uncoloured) copy of
$K_{1, \Delta}$ and joining each leaf to $d<\Delta$ distinct new
vertices, via edges coloured $1, 2, \ldots, d$. The resulting graph
$G$ has maximum degree $\Delta$, and contains a precoloured subgraph
$H$ with maximum degree $d$. However, in order to extend the
edge-precolouring to $G$, we need $\Delta$ new colours, meaning that
for a $(\Delta+t)$-edge-colouring of $G$, we need $d\leq t$.

\begin{figure}
\centering
\includegraphics[height=4.2cm]{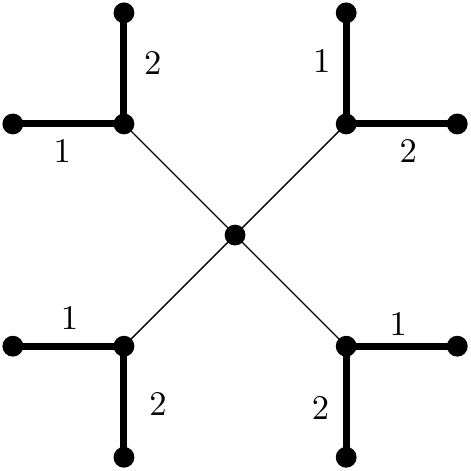}
\caption{A graph $G$ with maximum degree $\Delta=4$ and a precoloured subgraph of maximum degree $d=2$. In order to extend the edge-precolouring to a $(\Delta+t)$-edge-colouring of $G$ we need $t\geq d$.}
\label{tlessthank}
\end{figure}

If it happened that $H$ was edge-coloured efficiently (i.e. using at most $\chi'(H)$ colours), then our problem would be significantly reduced. In this special situation, one could use a completely new set of $\chi'(G-E(H))$ colours to extend to an edge-colouring of $G$ with at most the following number of colours (according to Vizing's Theorem):
\begin{equation}\label{Hefficient}
\chi'(G-E(H))+\chi'(H)\leq \chi'(G)+\chi'(H) \leq \Delta+ d+2.
\end{equation}
That is, when $H$ has been edge-coloured efficiently, the answer to Question \ref{qu} is \emph{yes} whenever $d \leq t-2$. Since extension can be impossible when $d>t$ (according to the above paragraph), this makes $d\in\{t-1, t\}$ the only interesting values in this case, with further restrictions if any of the inequalities in (\ref{Hefficient}) are strict. For example, if both $G$ and $H$ have chromatic index equal to their maximum degrees, then the colouring described above works whenever $d\leq t$, and hence we get a sharp threshold. Of course, this only works when $H$ has been edge-precoloured efficiently, and in general we have no control over the edge-precolouring on $H$.

While edge-colouring is in general an NP-hard problem, the situation is somewhat simpler for planar
graphs. For $\Delta=2,3,4,5$ there are examples of planar graphs with
chromatic index $\Delta$ and $\Delta+1$. However, every planar $G$
with $\Delta\geq 7$ is $\Delta$-edge-colourable; the case $\Delta=7$
was proved independently by Sanders and
Zhao \cite{SZ} and Zhang \cite{Zh}, and the case $\Delta\geq 8$ was
proved by Vizing, who conjectured it should also hold for 6 (as well
as the now-established 7). When focusing on planar graphs, there are
additional techniques at one's disposal, in particular the so-called
discharging method, that make edge-colouring easier.

We make progress on Question \ref{qu} in this paper by focusing on
planar graphs. In particular, we prove that the answer to Question
\ref{qu} is \emph{yes} whenever $d\leq t$, provided $d$ is small enough or $\Delta$ is large
enough. As discussed above, the $d\leq t$ assumption is sharp.
(In fact, we actually prove a stronger result involving
  list edge colouring, stated at the end of this section as Theorem~\ref{main},
  but more exposition is required to properly state and contextualize the
  stronger result.)

\begin{theorem}\label{maincor} Let $G$ be a planar graph of maximum
  degree at most $\Delta$, let $t$ be a positive integer, and let $H$ be a subgraph of $G$ that
  has been $(\Delta+t)$-edge-coloured.  If $H$ has maximum degree at
  most $d$, then the edge-precolouring can be extended to a
  $(\Delta+t)$-edge-colouring of $G$ provided that either:
  \begin{enumerate}
  \item $d \leq t-4$, or
  \item $t-3\leq d \leq t$ and
    \[ \Delta \geq
      \begin{cases}
        16 + d, & \text{if $d = t$}, \\
        9 + d, & \text{if $d = t-1$}, \\
        8 + d, & \text{if $d = t-2$}, \\
        7 + d, & \text{if $d = t-3$}.
      \end{cases} \]
  \end{enumerate}
\end{theorem}

Theorem \ref{maincor} does not include the case $t=0$, however the requirement of $d\leq t$ means that would correspond to $H$ being edgeless. Then the problem is not about precolouring at all, but simply about edge-colouring planar graphs as discussed above.

The case $d=t=1$ of Theorem \ref{main} was previously established by
Edwards, Gir\~{a}o, van den Heuvel, Kang, Sereni and the third author
\cite{EGVKPS}, with the slightly stronger assumption of
$\Delta\geq 19$. (Note that the restriction of our proof for Theorem
\ref{maincor} to this case provides a somewhat new proof; both arguments
use global discharging, but we discharge in a different way). After
the seminal work of Marcotte and Seymour \cite{MS}, the vertex-version
of the precolouring extension problem received much more attention
than Question \ref{qu}.  Edwards et al.~\cite{EGVKPS} re-initiated
this study in their paper, with planar graphs being only one of the
many families they considered. The main concern in \cite{EGVKPS}
however is when $H$ is a matching, and in order to guarantee
extensions they often impose distance conditions on the edges in the
precoloured matching. In particular, this means avoiding the issues
with $t$ being too small as exhibited in Figures
\ref{tlessthankminus1} and \ref{tlessthank}. Specifically, in addition
to the aforementioned result for $d=t=1$, they showed that if $H$ is
an edge-precoloured matching in a planar graph $G$ where edges are at distance at least 3
from one another, then any $\Delta$-edge-colouring on $H$ can be
extended to $G$ provided $\Delta\geq 20$. More recently,
Gir\~ao and Kang~\cite{girao-kang} studied extension from precoloured
matchings in general graphs, proving that if $H$ is a matching in a
(not necessarily planar) graph $G$ where
edges are distance at least $9$ from each other, then any $(\Delta+1)$-edge-colouring
on $H$ can be extended to a $(\Delta+1)$-edge-colouring of $G$.

As Edwards et al. \cite{EGVKPS} observed, extending an edge-colouring is closely
related to list-edge-colouring. An \emph{edge list assignment} on a graph $G$ is a function
  $L$ that assigns to each edge $e \in E(G)$ a list of colours $L(e)$.
  If $L$ is an edge list assignment on a graph $G$, an
  \emph{$L$-edge-colouring} of $G$ is an edge-colouring of $G$ such
  that every edge $e$ is given a colour from $L(e)$. Note that a
  classical $k$-edge-colouring of $G$ can be viewed as an
  $L$-edge-colouring for the list assignment $L$ defined by
  $L(e) = \{1,\ldots, k\}$ for all $e \in E(G)$. A graph $G$ is
  \emph{$k$-list-edge-colourable} if it is $L$-edge-colourable for every
  edge list assignment $L$ such that $\sizeof{L(e)} \geq k$ for all
  $e \in E(G)$. The notorious List-Edge-Colouring Conjecture
(attributed to many sources, some as early as 1975; see
\cite{toft-jensen}) asserts that every $G$ is
$\chi'(G)$-list-edge-colourable. If this conjecture is true, then
given the above discussion on the chromatic index of planar graphs,
$G$ \emph{should} be $\Delta$-list-edge-colourable whenever
$\Delta\geq 7$ (or perhaps 6). This has been verified when
$\Delta\geq 12$.

\begin{theorem}[Borodin, Kostochka, and Woodall
  \cite{BKW}] \label{thm:oldBKW} If $G$ is a planar graph with
  maximum degree $\Delta(G) \geq 12$, then $G$ is
  $\Delta(G)$-list-edge-colourable.
\end{theorem}

Borodin~\cite{borodin} proved a similar result; a short proof
  of this result was later obtained by Cohen and Havet~\cite{cohen-havet}.

  \begin{theorem}[Borodin~\cite{borodin}]\label{thm:Bor}
    If $G$ is a planar graph with maximum degree $\Delta(G) \geq 9$,
    then $G$ is $(\Delta(G)+1)$-list-edge-colourable.
  \end{theorem}

In the present paper we have in fact proved the list-edge-colouring analog of Theorem \ref{maincor}. This stronger result is as follows.

\begin{theorem}\label{main} Let $G$ be a planar graph of maximum
  degree at most $\Delta$, let $L$ be an edge list assignment on $G$
  with $\sizeof{L(e)} \geq \Delta+t$ for all $e \in E(G)$, where $t$
  is a positive integer, and let $H$ be a subgraph of $G$ that
  has been $L$-edge-coloured.  If $H$ has maximum degree at
  most $d$, then the edge-precolouring can be extended to an
  $L$-edge-colouring of $G$ provided that either:
  \begin{enumerate}
  \item $d \leq t-4$, or
  \item $t-3\leq d \leq t$ and
    \[ \Delta \geq
      \begin{cases}
        16 + d, & \text{if $d = t$}, \\
        9 + d, & \text{if $d = t-1$}, \\
        8 + d, & \text{if $d = t-2$}, \\
        7 + d, & \text{if $d = t-3$}.
      \end{cases} \]
  \end{enumerate}
\end{theorem}

We again omit the case $t=0$, however the required $d\leq t$ condition means that $H$ is edgeless and hence the best result is that of Theorem~\ref{thm:oldBKW} above. Theorem \ref{main} does have something meaningful to say when $H$ edgeless however: the case $t=1$ and $d=0$ gives Theorem \ref{thm:Bor} precisely.

The following section contains some technical results needed for our proof of Theorem \ref{main}, which comprises Section 3. The final section of this paper, Section 4, is about pushing Theorem \ref{main} beyond planar graphs. We show that requiring $G-E(H)$ to be planar is sufficient, and in fact ``planar'' can be replaced by ``non-negative Euler characteristic''.

\section{Technical Lemmas}
In this section, we gather some technical lemmas that will be
needed for the proof of Theorem~\ref{main}.

\begin{theorem}[Borodin, Kostochka, Woodall
  \cite{BKW}] \label{thm:bipedge} Let $G$ be a bipartite graph and let
  $L$ be an edge list assignment on $G$. If
  $\sizeof{L(xy)} \geq \max\{\deg(x), \deg(y)\}$ for every edge
  $xy \in E(G)$, then $G$ is $L$-edge-colourable.
\end{theorem}
Edwards et al.~\cite{EGVKPS} applied Theorem~\ref{thm:bipedge} to obtain
a precolouring extension result for bipartite graphs (Theorem~15 of
\cite{EGVKPS}), which we will use as part of our proof. While the result
as stated in \cite{EGVKPS} only applies to classical edge-precolouring,
a list-edge-colouring version can be obtained using essentially the
same proof:
\begin{theorem} \label{thm:bipextend} Let $G$ be a bipartite
  multigraph, and let $L$ be an edge list assignment on $G$ with
  $\sizeof{L(e)} \geq \Delta+t$ for all $e \in E(G)$. Let $H$ be a
  subgraph of $G$ that has been $L$-edge-coloured. If $H$ has maximum
  degree at most $d$, then the edge-precolouring can be extended to an
  $L$-edge-colouring of $G$ provided that $t \geq d$.
\end{theorem}
\begin{proof}
  Let $G' = G - E(H)$. For each edge $e \in E(G')$, let $L'(e)$ be
  obtained from $L(e)$ by removing all colours used on the
  edges of $H$ incident to $e$. Let $xy$ be an arbitrary
  edge of $G'$. Now
  \[ \sizeof{L'(xy)} \geq \sizeof{L(xy)} - \deg_H(x) - \deg_H(y) \geq \Delta + t - \deg_H(x) - \deg_H(y). \]
  Since $t \geq d \geq \Delta(H)$, this implies that
  \begin{align*}
    \sizeof{L'(xy)} &\geq \Delta - \deg_H(x), \quad \text{and} \\
    \sizeof{L'(xy)} &\geq \Delta - \deg_H(y).
  \end{align*}
  On the other hand,
  \begin{align*}
    \deg_{G'}(x) &= \deg_G(x) - \deg_H(x) \leq \Delta - \deg_H(x), \quad \text{and} \\
    \deg_{G'}(y) &= \deg_G(y) - \deg_H(y) \leq \Delta - \deg_H(y).
  \end{align*}
  Thus, $\sizeof{L'(xy)} \geq \max\{\deg_{G'}(x), \deg_{G'}(y)\}$, and this holds for all $xy \in E(G')$. By
  Theorem~\ref{thm:bipedge}, it follows that $G'$ is $L'$-edge-colourable,
  and any $L'$-edge-colouring of $G'$ gives the desired $L$-edge-colouring
  of $G$.
\end{proof}

In what follows and in the main argument, given a graph $G$,
we define $V_i(G)=V_i$ as the set of all vertices $v \in V(G)$ with
$\deg(v)=i$, and we define $V_{[a,b]}(G)=V_{[a,b]}$ as
$\cup_{i\in[a,b]}V_i$.
\begin{lemma}\label{cyclefix}
  Let $G$ be a graph of maximum degree at most $\Delta$, and
    let $L$ be an edge list assignment on $G$ with
    $\sizeof{L(e)} \geq \Delta + t$ for all $e \in E(G)$. Let $H$ be
  a subgraph of $G$ with maximum degree at most $d$. Suppose that $H$
  has been $L$-edge-coloured, and that this extends to an
  $L$-edge-colouring of $G-e$ for all
  $e\in E(G)\setminus E(H)$, but not to $G$.

  Let $A=V_{[a_0,a]}$ and $B=V_{[b_0,\Delta]}$, where $a_0,a,b_0$ are
  positive integers with $a_0 \geq  t+1 $, $b_0 > a$, and $ a+b_0 \geq  \Delta+t+1$. Let $X$
  be the bipartite subgraph of $G-E(H)$ induced by the bipartition
  $(A, B)$.  If every vertex $u\in A$ has the property that
  \[\deg_X(u) \geq \deg_G(u)-d,\]
  then
  \[(t+1-d)|A| \leq \sum_{i=b_0}^{\Delta} (a+i-1-(\Delta+t))|V_i|.\]
  Moreover, if $a_0>t+1$ and $a+b_0>\Delta+t+1$ then the above inequality is strict.
\end{lemma}
\begin{proof}
Say
  that an induced subgraph $J \subseteq X$ is \emph{bad} if
  \begin{itemize}
  \item $\deg_J(u) \geq \deg_G(u) - t$ for all $u \in A \cap V(J)$, and
  \item $\deg_J(v) \geq a + \deg_G(v) - (\Delta+t)$ for all $v \in B \cap V(J)$.
  \end{itemize}
  Notice that for all $u\in A$, $v\in B$,
  \begin{equation}\label{eq:A}\deg_G(u) - t \geq a_0 - t \geq 1\end{equation} and
  \begin{equation}\label{eq:B}a + \deg_G(v) - (\Delta+t) \geq a + b_0 - (\Delta +t) \geq  1,\end{equation}
  so that if a bad induced subgraph exists, it has no isolated vertices,
  and in particular has at least one edge.  We will first show that $X$ has no
  bad induced subgraph, and then show that this implies the desired
  claim.

  Suppose that $X$ has a bad induced subgraph $J$. Let $G' = G-E(J)$.
  Since $E(J)$ is nonempty, $G'$ is a proper subgraph of
  $G$, so by assumption, the edge-precolouring on $H$ extends to an
  $L$-edge-colouring $\varphi$ of $G'$. We derive a
  contradiction by showing we can further extend to an
  $L$-edge-colouring of $G$. To this end, let
    $L^J$ be the edge list assignment on $J$ defined as follows: for each edge $uv \in E(J)$,
    $L^J(uv)$ is the set of colours from $L(uv)$ that do not appear on
    any $G'$-edge adjacent to $uv$. Observe that for each
  $uv \in E(J)$, we have
  \[
  \sizeof{L^J(uv)} \geq \Delta+t - \deg_G(u) - \deg_G(v) + \deg_J(u) + \deg_J(v).
  \]
  Since $J$ is bad, we have $\deg_J(u) \geq \deg_G(u) - t$, so that
  \begin{align*}
    \sizeof{L^J(uv)} &\geq \Delta - \deg_G(v)+ \deg_J(v)
    \geq \deg_J (v),
  \end{align*}
  and likewise $\deg_J(v) \geq \deg_G(v) + a - (\Delta+t)$ so that
  \begin{align*}
    \sizeof{L^J(uv)} &\geq a - \deg_G(u) + \deg_J(u)\geq \deg_J(u).
  \end{align*}
  Hence, for every $uv \in E(J)$, we have
  $\sizeof{L^J(uv)} \geq \max\{d(u), d(v)\}$. By
  Theorem~\ref{thm:bipedge}, $J$ is $L^J$-edge-colourable. Now any
  proper $L^J$-edge-colouring of $J$, combined with the
    $L$-edge-colouring $\varphi$ of $G'$, yields a proper
  $L$-edge-colouring of $G$ that extends the edge-precolouring of $H$
  as desired; contradiction.

  Hence, $X$ contains no bad induced subgraph, and so every induced
  subgraph $J$ of $X$ contains a vertex violating the definition of a
  ``bad'' subgraph.  By iteratively removing these vertices and
  counting the edges removed when each vertex is deleted, we see that
  \begin{align*}
    \sizeof{E(X)} &\leq \sum_{u \in A}[\deg_G(u) - t - 1] + \sum_{v \in B}[a + \deg_G(v) - (\Delta+t) - 1] \stepcounter{equation}\tag{\theequation}\label{eq:thecount}\\
    &\leq \sum_{u \in A}[(\deg_X(u)+d)-t - 1] + \sum_{v \in B}[a + \deg_G(v) - (\Delta+t) - 1] \\
    &= \sizeof{E(X)} + \sum_{u \in A}[d -t- 1] + \sum_{i=b_0}^{\Delta} (a + i - (\Delta+t) - 1)\sizeof{V_i }.
  \end{align*}
  Rearranging the last inequality yields
  \[ (t+1-d)\sizeof{A} \leq \sum_{i=b_0}^{\Delta} (a + i -1- (\Delta+t))\sizeof{V_i}, \]
  which is the desired conclusion. If we additionally know that  $a_0>t+1$ and $a+b_0>\Delta+t+1$, then inequalities (\ref{eq:A}) and (\ref{eq:B}) become strict. Hence each $u\in A$ and $v\in B$ is contributing a positive amount to the right-hand-side of (\ref{eq:thecount}). Since the last vertex removed is isolated, this is an overcount, and hence we get a strict inequality.
\end{proof}
\section{Proof of Theorem~\ref{main}}
For fixed values of $\Delta, t, d$, we choose a counterexample
$(G, H)$ where the quantity $3\sizeof{E(G)} + \sizeof{V_{[2,t+1]}(G)}$
is as small as possible.

\begin{claim}\label{Gminuse} The edge-precolouring on $H$ can be
  extended to an $L$-edge-colouring of $G-e$ for any
  $e\in E(G)\setminus E(H)$.
\end{claim}
\begin{proofc}
  Let any $e \in E(G) \setminus E(H)$ be given, and let $G' =
  G-e$. Note that $(G', H)$ satisfies the hypotheses of the theorem
  with $\Delta, t, d$. Exactly two vertices in $G'$ have lower degrees
  than in $G$, so $\sizeof{V_{[2,t+1]}(G')}$ may be as large as
  $\sizeof{V_{[2,t+1]}(G)}+2$. However, since $G'$ has one edge less
  than $G$, we still get that
  \[3\sizeof{E(G')} + \sizeof{V_{[2,t+1]}(G')} < 3\sizeof{E(G)} + \sizeof{V_{[2,t+1]}(G)}.\]
  Hence, by our choice of counterexample, the edge-precolouring of $H$
  extends to an $L$-edge-colouring of $G'$.
\end{proofc}

\begin{claim}\label{degsum}
  If $uv \in E(G)\setminus E(H)$, then $\deg_G(u)+\deg_G(v) \geq \Delta+t+2$.
\end{claim}
\begin{proofc} By Claim \ref{Gminuse}, the edge-precolouring of $H$
  can be extended to an $L$-edge-colouring $\varphi$ of $G-uv$. The
  edge $uv$ sees at most $\deg_G(u)+\deg_G(v)-2$ different colours in
  $\varphi$, so since $(G, H, t)$ is a counterexample, it must be that
  $\deg_G(u)+\deg_G(v)-2 \geq \Delta+t$.
\end{proofc}
\begin{claim}\label{allinH} If $v \in V_{[1,t+1]}$, then every edge incident to $v$
  in $G$ is also in $H$.
\end{claim}
\begin{proofc}
  Assume for contradiction that $v \in V_{[1, t+1]}$ and $v$ is incident to an edge
  not in $H$, say $uv$. By Claim~\ref{degsum}, we know that $\deg_G(u) + \deg_G(v) \geq \Delta + t + 2$.
  However, since $\deg_G(v) \leq t+1$, this implies that $\deg_G(u) \geq \Delta+1$, a contradiction.
\end{proofc}
\begin{claim}\label{2tplus1empty}
  $V_{[2,t+1]}=\emptyset$.
\end{claim}
\begin{proofc}
  Suppose not, and take $v \in V_{[2,t+1]}$. By Claim~\ref{allinH},
  every edge $uv$ incident to $v$ must lie in $H$.

  Let $G'$ and $H'$ be the graphs obtained from $G$ and $H$,
  respectively, by deleting $v$ and, for each $u \in N_G(v)$, adding a
  new vertex $v_u$ adjacent only to $u$. We precolour each edge $uv_u$
  with the same colour received by the edge $uv$ in the precolouring of $H$. See Figure
  \ref{reduction}. Observe that the edge-precolouring of $H'$ extends to $G'$ if and only if
  the edge-precolouring of $H$ extends to $G$.

  \begin{figure}
    \centering
    \includegraphics[height=3.5cm]{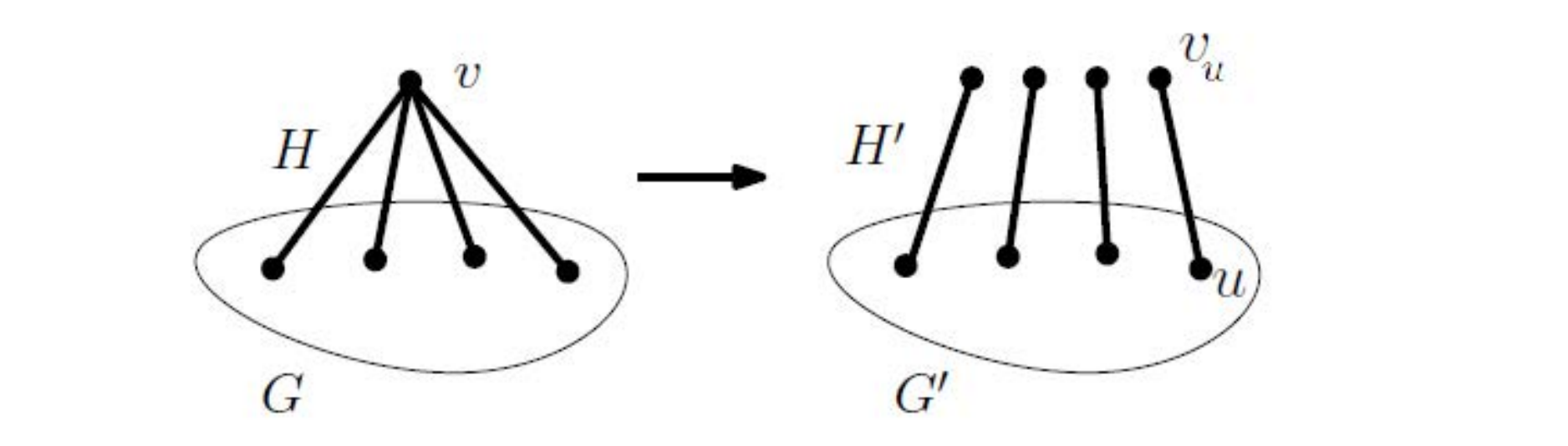}
    \caption{Moving from $(G,H)$ to $(G', H')$ in the proof of Claim \ref{2tplus1empty}.}
    \label{reduction}
  \end{figure}

  Now $G'$ has the same number of edges as $G$, and has one fewer
  vertex in $V_{[2,t+1]}$. As $\Delta(G') \leq \Delta$ and
  $\Delta(H') \leq d$, our choice of counterexample implies that
  the edge-precolouring of $H'$ extends to $G'$, but this means that
  the edge-precolouring of $H$ extends to $G$ as well.
\end{proofc}

\begin{claim}\label{verts} Every vertex of $G$ is either a leaf
  incident to an edge in $H$, or of degree at least $t+2$.
\end{claim}

\begin{proof} This follows by combining  Claim \ref{allinH} and Claim~\ref{2tplus1empty}.
\end{proof}

Let $F_m$ be the set of faces in $G$ with exactly $m$ vertices on its boundary having degree 3 or higher in $G$.

\begin{claim}\label{F3}
  $F_0 = F_1 = F_2 = \emptyset$.
\end{claim}
\begin{proofc} Suppose that $f\in F_0\cup F_1\cup F_2$; we will show a
  contradiction. We know that $V_2=\emptyset$ by Claim \ref{verts},
  since $t\geq 1$. So, if the boundary of $f$ contains a cycle,
  then it contains at least three vertices of degree at least three,
  yielding a contradiction. Thus, the boundary of $f$ contains
  no cycle. This means that $G$ is a forest, and $f$ is its one
  face. In particular, $G$ is bipartite. By Theorem~\ref{thm:bipextend},
  this implies that the precolouring of $H$ extends to all of $G$,
  contradicting our choice of $G$ as a counterexample.
\end{proofc}

We now introduce a discharging argument. To each vertex in $G$ assign
an initial charge of $\alpha(v)=3\deg_G(v)-6$. To each face in $G$
assign an initial charge of $\alpha(f)=-6$. We also define an
additional structure $P$ (a ``global pot'') and assign to it an
initial charge of $\alpha(P)=0$. We discharge along the following
rules:
\begin{enumerate}
\item[(a)]  For each $m$, every face $f \in F_m$ takes $\frac{6}{m}$ from each
  vertex of degree 3 or higher on its boundary.
\item[(b)] Every vertex $v \in V_1$ takes $3$ from its neighbor.
\end{enumerate}
In the special case where $t=d+\ell$ for $\ell\in\{0, 1, 2, 3\}$, we also add the following rules:
\begin{enumerate}
\item[(c)] For every vertex $v \in V_i$, where $i\in\{t+2, \ldots, t+5-\ell\}$:\\ \hspace*{1.5in} $v$ takes $t+6-\ell-i$ from $P$.
\item[(d)] For every vertex $v\in V_j$, where $j\in\{\Delta-3+\ell, \ldots, \Delta\}$:\\ \hspace*{1.5in} $v$ gives $\frac{q(j)(q(j) + 1)}{2(\ell+1)}$ to $P$,
 where $q(j) = j-\Delta+4-\ell$.
\end{enumerate}

While it is not immediately obvious, discharging rules (c) and (d) never apply to the same vertex, due to the following claim.

\begin{claim} \label{disjoint} If $t=d+\ell$ for some $\ell\in\{0, 1, 2, 3\}$, then $\Delta-3+\ell>t+5-\ell$.
\end{claim}

\begin{proofc}  We get the desired inequality if and only if $\Delta+2\ell > 8+t$. If $\ell = 0$,
  then we have $d=t$, so the hypothesis of Theorem~\ref{main} yields
  \[ \Delta+2\ell=\Delta\geq 16+d=16+t>8+t. \]
 If $\ell \in \{1,2,3\}$ we may rewrite hypothesis of Theorem~\ref{main} as 
    \[\Delta \geq 10+d-\ell=10+(t-\ell)-\ell=10+t-2\ell, \textrm{\ so}\]
    \[ \Delta+2\ell \geq 10+t> 8+t. \]
\end{proofc}

Using Euler's formula for planar graphs, the sum of initial charges is at most $-12$:
\begin{align*}
  \alpha(P)+ \sum_{v \in V(G)} \alpha(v) + \sum_{f \in F(G)} \alpha(f) &=0+\sum_{v \in V(G)} (3\deg_G(v) -6) + \sum_{f \in F(G)} (-6)\\
                                                                       &= 6|E(G)| - 6|V(G)| - 6|F(G)| \leq 6 (-2)= -12.
\stepcounter{equation}\tag{\theequation}\label{eq:Euler}
\end{align*}
For each graph element $x$ (either a vertex, a face, or the global
pot), let $\alpha'(x)$ denote the final charge of $x$. Since each
discharging rule conserves the total charge, we see that
$\sum_x \alpha'(x) = \sum_x \alpha(x) = -12$.  We will achieve our
desired contradiction by showing that the final charge of each element
is nonnegative.

First consider a face $f$. By Claim \ref{F3}, $f\in F_m$ for
$m\geq 3$. So according to discharging rule (a) (the only rule
affecting $f$),
\[\alpha'(f)=(-6)+ m (\tfrac{6}{m})=0.\]

Now consider the global pot $P$. We know $\alpha(P)=0$ and that the charge of $P$ is unaffected when $d \leq t-4$, so the following claim precisely amounts to showing showing that $\alpha'(P)> 0$ when $t-3\leq d \leq t$.

\begin{claim}\label{coef} If $t=d+\ell$ for some $\ell\in\{0,1,2,3\}$, then
  \begin{equation}\label{coefs}\sum_{i=t+2}^{t+5-\ell} (t+6-\ell-i)|V_i|< \sum_{j=\Delta-3+\ell}^{\Delta}\frac{q(j)(q(j)+1)}{2(\ell+1)} |V_j|.\end{equation}
\end{claim}

\begin{proofc}
  For each $k\in\{0, \ldots, 3-\ell\}$, define
  $A_k=V_{[t+2, t+5-\ell-k]}$ and $B_k=V_{[\Delta-3+\ell+k, \Delta]}$
  and let $X_k$ be the bipartite subgraph of $G-E(H)$ induced by the
  partition $(A_k, B_k)$. We will show we can apply Lemma \ref{cyclefix}
  for each value of $k$, and then we will sum the resulting inequalities
  to get our desired result. For fixed $k$, this means we want to
  apply Lemma \ref{cyclefix} with parameter choices
  \begin{align*}
    a_0 &= t+2, & a &= t+5-\ell-k, \\
    b_0 &= \Delta-3+\ell+k,
  \end{align*}

 \noindent and hence to do so we must verify that $a_0\geq t+1$ (true) and that
  $a+b_0\geq \Delta+t+1$, which is true since
  \[(t+5-\ell-k)+(\Delta-3+\ell + k)=t+2+\Delta.\]

  In fact, since both these inequalities hold strictly, we will apply the strict version of Lemma \ref{cyclefix}. Of course, there are several other hypotheses we must check. In particular, we must verify that $b_0 > a$, which is equivalent to showing
    that $\Delta > t + 8 - 2\ell$. Since $t = d+\ell$, we get this inequality by Claim \ref{disjoint}.
  By Claim \ref{Gminuse}, we can therefore apply Lemma \ref{cyclefix} for $k$ provided that every vertex $u\in A_k$ has the property that
  $$\deg_{X_k}(u) \geq \deg_G(u)-d.$$
  Consider such a vertex $u$ with incident edge $uv$ in $E(G)\setminus E(H)$. Since $u\in A_k$, and by Claim \ref{degsum}, we know that
  $$\deg_G(v)\geq \Delta+t+2-\deg_G(u)\geq \Delta+t+2-(t+5-\ell-k)=\Delta-3+\ell+k.$$
  This means, by definition of $X_k$, that the edge $uv$ is in $X_k$.  So we get $deg_{X_k}(u)\geq deg_G(u)-deg_H(u) \geq deg_G(u)-d$, as desired.

  For any fixed $k$, we can now apply Lemma \ref{cyclefix} to get
  \begin{equation}\label{eq:fixedcycle} (\ell+1) |A_k| < \sum_{j=\Delta-3+\ell+k}^{\Delta} (q(j)-k)|V_j|,
  \end{equation}
  since $t+1-d = \ell+1$ by the hypothesis of Claim~8, and since, for our choices of parameters,
  \begin{align*}
    a+j-1 - (\Delta+t) &= (t+5-\ell-k)+j - 1 - (\Delta + t) \\
                       &= j - \Delta + 4 - \ell - k \\
                       &= q(j) - k.
  \end{align*}
  Dividing \eqref{eq:fixedcycle} by $(\ell+1)$ and summing over all $k$ yields
  \begin{equation}\label{bigsum}
    \sum_{k=0}^{3-\ell}|A_k| < \left(\tfrac{1}{\ell+1}\right)\sum_{k=0}^{3-\ell}\sum_{j=\Delta-3+\ell+k}^{\Delta} (q(j)-k)|V_j|.
  \end{equation}
  The left-hand-side of (\ref{bigsum}) is
  \begin{eqnarray*}\sum_{k=0}^{3-\ell}|V_{[t+2, t+5-\ell-k]}| &=& |V_{[t+2, t+5-\ell]}|+|V_{[t+2, t+4-\ell]}|+\cdots+|V_{[t+2, t+2]} |\\
                                                              &=& (4-\ell)|V_{t+2}|+\cdots +2|V_{t+4-\ell}|+ |V_{t+5-\ell}|\\
                                                              &=& \sum_{i=t+2}^{t+5-\ell}(t+6-\ell-i)|V_i|,
  \end{eqnarray*}
  matching the left-hand side of (\ref{coefs}). It remains only to show that the right-hand-side of (\ref{bigsum}) equals the right-hand side of (\ref{coefs}). To this end, note that
  \[
  j\geq \Delta-3+\ell+k \iff k\leq j-\Delta+3-\ell=q(j)-1, \textrm{and so}
  \]
  \[
    \sum_{k=0}^{3-\ell}\sum_{j=\Delta-3+\ell+k}^{\Delta} (q(j)-k)|V_j| = \sum_{j=\Delta-3+\ell}^{\Delta} \left(\sum_{k=0}^{q(j)-1}(q(j)-k)\right)|V_j|.
  \]
  Now the bracketed sum can be rewritten as
  \[ \sum_{k=0}^{q(j)-1} (q(j)-k) =q(j)+(q(j)-1)+(q(j)-2)+\cdots +1= \tfrac{q(j)(q(j)+1)}{2}, \]
  which is precisely what we needed to prove.
\end{proofc}

We have now shown $\alpha'(P)>0$, so it remains only to consider the final charge of an arbitrary vertex $v$. If $v\in V_1$, then only discharging rule (b) affects $v$, and we get
$$\alpha'(v)=(-3)+3=0.$$
By Claim \ref{verts}, we may now assume that $\deg_G(v)\geq t+2$.

Suppose $v$ lies on the boundary of $x$ distinct faces and is incident to $y$ leaves. We know that $x$ is no more than  $\deg_G(v)-y$,  so $x+y \leq \deg_G(v)$. We also know that $y \leq d$, by Claim \ref{verts} and by definition of $d$. By doubling the first inequality and adding the result to the second inequality we get
\begin{equation}\label{2x} 2x+3y \leq 2\deg_G(v)+d.
\end{equation}
Since $F_0,F_1,F_2 = \emptyset$ by Claim \ref{F3}, each of the $x$ distinct faces incident to $v$ has at least $3$ vertices of degree at least 3 on their boundary. This means that each of these $x$ faces takes charge at most  2 from $v$, according to discharging rule (a). Each of the $y$ leaves incident to $v$ takes exactly 3 from $v$, according to discharging rule (b). Hence by inequality (\ref{2x}), after applying discharging rules (a) and (b) (but before considering discharging rules (c) or (d)), the charge of $v$ is at least
\begin{equation}\label{ab}
  3\deg_G(v)-6 -(2x+3y) \geq \deg_G(v)-6-d.
\end{equation}
Note that since $d\leq t$, the additional discharging rules (c) and (d) are applied precisely when $d \geq t-3$. If $d\leq t-4$, then we do not apply them, and by inequality (\ref{ab}),
$$\alpha'(v)\geq \deg_G(v)-6-d \geq \deg_G(v)-6- (t-4)=\deg_G(v)-(t+2)\geq 0.$$

We may now assume that $t=d+\ell$ for $\ell \in\{0,1,2,3\}$. Let $p$ denote the total charge transferred from $P$ to $v$ according to discharging rules (c) and (d); note that $p$ may be positive, negative, or zero.
In all cases, by inequality (\ref{ab}), we have that
\begin{equation}\label{finalp}
  \alpha'(v) \geq \deg_G(v)-6-d +p.
\end{equation}

If neither discharging rule (c) nor (d) applies to $v$, then we know that $t+5-\ell < \deg_G(v)$ and therefore (\ref{finalp}) says that
$$\alpha'(v) \geq (t+5-\ell+1)-6-d +(0)=(t-d)-\ell=0,$$
as desired.

Now suppose that discharging rule (c) applies to $v$ (and hence (d) does not, according to Claim \ref{disjoint}). In this situation, \eqref{finalp} implies that
\[\alpha'(v) \geq \deg_G(v)-6-d +(t+6-\ell-\deg_G(v))=0.\]

Finally, we may assume that discharging rule (d) applies to $v$ (and
hence (c) does not, according to Claim \ref{disjoint}). In this case,
we have $t = d+\ell$, where $\ell \in \{0,1,2,3\}$, andF
$\deg_G(v) \in \{\Delta-3+\ell, \ldots, \Delta\}$. By~\eqref{finalp},
\[ \alpha'(v) \geq \deg_G(v) - 6 - d - \left(\frac{ (\deg_G(v) - (\Delta - 4 + \ell))(\deg_G(v) - (\Delta - 5 + \ell))}{2(\ell+1)} \right). \]
Writing $\deg_G(v)$ as $\Delta - h + \ell$, where $h \in \{\ell,\ldots,3\}$, we can rewrite this lower bound as
\begin{align*}
  \alpha'(v) &\geq \Delta - h + \ell - 6 - d - \left( \frac{(4-h)(5-h)}{2(\ell+1)}\right) \\
             &= \Delta - d - \left(6 + h - \ell + \frac{(4-h)(5-h)}{2(\ell+1)}\right).
\end{align*}
Table~\ref{tab:alphaprime} computes the bracketed quantity for each permissible combination of $\deg_G(v)$ and $\ell$.
For each possible value of $\ell$, the hypothesis of Theorem~\ref{main} ensures that this lower bound
  is always nonnegative.

\begin{table}
  \centering
  \[\begin{array}{r||c|c|c|c}
      & \ell=0 & \ell=1 & \ell=2 & \ell=3 \\\hline\hline
      \deg_G(v) = \Delta-3+\ell & \Delta - d - 10 & \Delta - d - \lfrac{17}{2} & \Delta - d - \lfrac{22}{3} & \Delta - d - \lfrac{25}{4}  \\\hline
      \deg_G(v) = \Delta-2+\ell & \Delta - d - 11 & \Delta - d - \lfrac{17}{2} & \Delta - d - 7 & *   \\\hline
      \deg_G(v) = \Delta-1+\ell & \Delta - d - 13 & \Delta - d - 9 & * & * \\ \hline
      \deg_G(v) = \Delta-0+\ell & \Delta - d - 16 & * & * & *
    \end{array}\]
  \caption{Lower bounds on $\alpha'(v)$ when discharging rule (d) applies. Starred entries are impossible due to $\deg_G(v) \leq \Delta$.}
  \label{tab:alphaprime}
\end{table}
We have proved that $\alpha'(x) \geq 0$ for every graph element $x$, and this completes the proof of Theorem~\ref{main}.

\section{Beyond planarity}

In the proof of Theorem \ref{main} our initial charges sum to at most $-12$, and after discharging the vertices, faces, and global pot all have nonnegative charge. In fact, when we examine inequality (\ref{eq:Euler}), we  see that the sum of initial charges is at most $-6\varepsilon$, where $\varepsilon$ is the Euler characteristic of the plane. Hence our argument works for any surface of positive Euler characteristic; namely $G$ may be embedded on the plane or projective plane. Moreover, this embedding requirement need not concern the edges of the precoloured $H$: imagine applying Theorem \ref{main} to the graph obtained by replacing every edge $e=uv$ in $H$ with a pair of edges $e_u=uu'$ and $e_v=vv'$ where $u', v'$ are new leaves, and $e_u$ and $e_v$ retain the precolouring (and lists) of $e$. Given these observations, we can strengthen Theorem \ref{main} by removing the assumption that ``$G$ is planar'' and replacing it by the somewhat milder ``$G-E(H)$ can be embedded in a surface of positive Euler characteristic''.

\section*{Acknowledgements}

We are indebted to an anonymous referee whose insightful comments strengthened our main result.\\

\bibliographystyle{amsplain}
\bibliography{bibEdgePrecolouringPlanar}
\end{document}